\pdfoutput=1
\documentclass[11pt,oneside,reqno]{amsart}
\usepackage[paperheight=279mm,paperwidth=18cm,textheight=26cm,textwidth=14cm]{geometry}
\usepackage[T1]{fontenc}
\usepackage[utf8]{inputenc}
\usepackage[unicode]{hyperref}
\hypersetup{
bookmarks=true,
colorlinks=true,
citecolor=[rgb]{0,0,0.5},
linkcolor=[rgb]{0,0,0.5},
urlcolor=[rgb]{0,0,0.75},
pdfpagemode=UseNone,
pdfstartview=FitH,
pdfdisplaydoctitle=true,
pdftitle={Reconstruction theorem},
pdfauthor={Pavel Zorin-Kranich},
pdflang=en-US
}

\usepackage{amsmath}
\usepackage{amssymb}
\usepackage{amsthm}
\usepackage{xspace}
\usepackage{mathtools}
\usepackage{esint}

\usepackage[style=alphabetic,maxalphanames=4]{biblatex}
\numberwithin{equation}{section}
\theoremstyle{plain}
\newtheorem{theorem}{Theorem}[section]

\newtheorem{lemma}[theorem]{Lemma}

\DeclarePairedDelimiter\abs{\lvert}{\rvert}

\def\phi{\varphi}
\def\dilate#1{(#1)}
\DeclarePairedDelimiterXPP\EE[1]{\E}{\lparen}{\rparen}{}{\renewcommand\given{\SetSymbol[\delimsize]}#1} 
\providecommand\given{}
\newcommand\SetSymbol[1][]{%
\nonscript\:#1\vert
\allowbreak
\nonscript\:
\mathopen{}}
\DeclarePairedDelimiterX\Set[1]\{\}{%
\renewcommand\given{\SetSymbol[\delimsize]}
#1
}

\newcommand{\one}{\mathbf{1}}
 
\newcommand*{\N}{\mathbb{N}}
\newcommand*{\Z}{\mathbb{Z}}
\newcommand*{\R}{\mathbb{R}}

\newcommand{\dif}{\mathop{}\!\mathrm{d}} 

\def\calB{\mathcal{B}}
\def\calD{\mathcal{D}}
\def\calN{\mathcal{N}}
\DeclareMathOperator{\supp}{supp}

\title[Reconstruction theorem]{Reconstruction theorem in quasinormed spaces}
\author[P.~Zorin-Kranich]{Pavel Zorin-Kranich}
\address{Mathematical Institute\\ University of Bonn}
\email{pzorin@uni-bonn.de}

\makeatletter
\@namedef{subjclassname@2020}{%
  \textup{2020} Mathematics Subject Classification}
\makeatother

\begin{document}

\begin{abstract}
We extend the Hairer reconstruction theorem for distributions due to Caravenna and Zambotti \cite{MR4261666} to general function spaces satisfying a translation and scaling condition.
This includes Besov type spaces with exponents below $1$ and Triebel--Lizorkin type spaces.
\end{abstract}
\maketitle

\section{Introduction}

Hairer's reconstruction theorem, originally proved in \cite[Theorem 3.10]{MR3274562} in the framework of regularity structures, gives conditions under which one can construct a distribution on $\R^{d}$ with a given local behavior.
Similarly to the sewing lemma in rough integration, see e.g.\ \cite[Lemma 4.2]{MR4174393}, a reconstruction theorem for distributions, that does not require regularity structures in its statement, was developed in \cite{MR4261666}.
In this article, we isolate the main estimate used in \cite{MR4261666} (Theorem~\ref{thm:local-reconstruction}) and formulate sufficient conditions on function space quasinorms for which the reconstruction theorem for distributions holds (Section~\ref{sec:fun-spaces}).

A Besov space version of the reconstruction theorem for regularity structures was proved in \cite{MR3684891}.
Our sufficient conditions cover Besov spaces with exponents $p,q\in [1,\infty]$, in which case the result has been independently obtained in \cite{arxiv:2106.12528}, Triebel--Lizorkin spaces with exponents $p,q \in (1,\infty)$, as well as some Besov spaces with exponents below $1$.
For Besov spaces with exponents below $1$, we observe additional restrictions on the integrability and regularity exponents.
Although we do not know if they are always necessary, in one dimension they are consistent with the restrictions that appear in the sewing lemma (see Section~\ref{sec:sewing}).

We begin with the formulation of the main pointwise estimate for reconstruction.
As test functions, we will use Hölder functions of some order $r > 0$.
Let $\tilde{r}\in\N$ with $\tilde{r} < r \leq \tilde{r}+1$ be the largest integer strictly smaller than $r$, and let
\begin{equation}
\label{eq:r-Holder}
\calB_{r} := \Set{ \xi\in C^{\infty}_{c}(\R^{d}) \given \supp\xi \subseteq B(0,1/2), \forall x,x' \in \R^{d} \ \abs{D^{\tilde{r}}\xi(x)-D^{\tilde{r}}\xi(x')} \leq \abs{x-x'}^{r-\tilde{r}} },
\end{equation}
For a function $\xi$ defined on $\R^{d}$, we denote by
\begin{equation}
\label{eq:dilate}
\xi_{x}^{\dilate{k}}(y) := 2^{dk} \xi(2^{k}(y-x))
\end{equation}
an $L^{1}$ scaled version of $\xi$ at scale $2^{-k}$ centered at $x$.
We also use the conventions $\xi^{\dilate{k}} := \xi^{\dilate{k}}_{0}$ and $\tilde{\xi}(y)=\xi(-y)$.
We denote integral averages by $\fint_{z \in Z} = \abs{Z}^{-1} \int_{z\in Z}$.
We write $a \lesssim b$ if there is a constant $C<\infty$, depending only on parameters $\gamma,\alpha,r,\phi,p,q,\calN$, such that $a \leq C b$.

\begin{theorem}[Local reconstruction]
\label{thm:local-reconstruction}
Let $\phi\in C^{\infty}_{c}(\R^{d})$ with $\supp \phi \subseteq B(0,1/2)$ and $\int\phi = 1$.
Let $(F_{x})_{x\in\R^{d}}$ be a Borel measurable family of distributions on $\R^{d}$.
Let $x\in\R^{d}$, $k\in\Z$, and $r>0$ be such that
\begin{equation}
\label{eq:local-coherence}
\begin{split}
H(k,x) &:=
\sum_{l\geq 0} 2^{-lr} \fint_{\abs{x-y} \leq 2^{-k}} \abs[\big]{(F_{y}-F_{x})(\phi_{y}^{\dilate{k+l}})} \dif y
\\ &+ \sum_{l\geq 0} \fint_{\abs{z-x} \leq 2^{-k}} \fint_{\abs{y-z} \leq 2^{-k-l}} \abs[\Big]{(F_{z}-F_{y})(\phi_{y}^{\dilate{k+l}})} \dif y \dif z
< \infty.
\end{split}
\end{equation}
Let $\rho$ be a smooth function, supported on $B(0,1/2)$, depending only on $\phi,r$, defined in \eqref{eq:diffconv}.
Then, the functions
\begin{equation}
\label{eq:fn}
f_n(z) := F_z(\rho^{\dilate{n}}_z),
\quad
n\in\N,
\end{equation}
are integrable on $B(x,2^{-k})$ for $n\geq k$ and converge, in the sense of distributions, as $n\to\infty$, to a distribution $f\in \calD'(B(x,2^{-k}))$ such that
\begin{equation} \label{eq:loc-rec-bd}
\sup_{\xi \in \calB_{r}} \abs{ (f-F_x)(\xi_x^{\dilate{k}})}
\lesssim H(k,x).
\end{equation}
\end{theorem}
The condition \eqref{eq:loc-rec-bd} clearly does not determine the distribution $f$ uniquely; for instance, it would also hold with $F_{x}$ in place of $f$.
The main point of Theorem~\ref{thm:local-reconstruction} is that the construction of $f$ does not depend on $x,k$.
In particular, if $H(x,k) < \infty$ for all pairs $(x,k)$ in some set $X$, then we obtain a distribution $f$ on the open set $\cup_{(x,k)\in X} B(x,2^{-k})$, and the estimate \eqref{eq:loc-rec-bd} holds for all $(x,k) \in X$.

The condition \eqref{eq:local-coherence} is a version of the \emph{coherence} condition in \cite[Definition 4.3]{MR4261666}, it expresses the idea that the family $F$ varies sufficiently slowly.
We formulate some function space consequences of Theorem~\ref{thm:local-reconstruction} in Section~\ref{sec:fun-spaces}.

Theorem~\ref{thm:local-reconstruction} is proved in Section~\ref{sec:proof}.
In Section~\ref{sec:sewing}, we compare the reconstruction theorem in Besov spaces (a special case of Theorem~\ref{thm:reconstruction}) and the sewing lemma in Besov spaces.
In Section~\ref{sec:negative-regularity}, we comment on a version of the reconstruction theorem for spaces with negative regularity.

\subsection*{Acknowledgment}
I thank the anonymous referees for detailed reports that helped to improve this article.

\section{Function spaces}
\label{sec:fun-spaces}
Let us now describe a class of function spaces in which the condition \eqref{eq:local-coherence} can be conveniently verified.
For simplicity, we consider function spaces on $\R^{d}$.
One could state similar results for open domains in $\R^{d}$, but it is difficult to foresee how much flexibility regarding the boundary behaviour would be useful.

We consider quasinorms $\calN$ defined on measurable functions $H : \N \times \R^{d} \to \overline{\R_{\geq 0}} = \R_{\geq 0} \cup \Set{\infty}$ and taking values in $\overline{\R_{\geq 0}}$.
A \emph{quasinorm} is a functional that is homogeneous and quasisubadditive:
\[
\calN(\lambda H) = \abs{\lambda} \calN(H)
\quad\text{and}\quad
\calN(H_{1}+H_{2}) \lesssim \calN(H_{1}) + \calN(H_{2}),
\]
for any $\lambda\in\R_{\geq 0}$ and any arguments $H,H_{1},H_{2}$.

The main assumption on $\calN$ is a scaling condition: there exists $\gamma > 0$ such that, for every $l\in\N$, we have
\begin{equation}
\label{eq:N-scaling}
\calN_{k,x} \fint_{\abs{z-x}\leq 2^{-k}} H(k+l,z) \dif z
\lesssim
2^{-l\gamma} \calN_{k,x} H(k,x).
\end{equation}

In addition to the above quantitative assumption, we also make the qualitative assumptions that $\calN$ satisfies a version of the monotone convergence theorem:
\begin{equation}
\label{eq:N-monotone-conv}
H_{n} \nearrow H
\implies
\calN_{k,x} H_{n}(k,x) \to \calN_{k,x} H(k,x),
\end{equation}
and that finiteness of $\calN H$ implies finiteness of the function $H$:
\begin{equation}
\label{eq:N-ae}
\calN_{k,x}H(k,x) < \infty
\implies
H<\infty \text{ a.e.}
\end{equation}

\begin{theorem}[Reconstruction]
\label{thm:reconstruction}
Let $\gamma>0$, $\alpha\geq 0$, and $r>0$ with $r > \alpha$.
Let $\calN$ be a quasinorm satisfying the conditions \eqref{eq:N-scaling}, \eqref{eq:N-monotone-conv}, \eqref{eq:N-ae}.

Let $\phi\in C^{\infty}_{c}(\R^{d})$ with $\supp \phi \subseteq B(0,1/2)$ and $\int\phi = 1$.
Let $(F_{x})_{x\in\R^{d}}$ be a Borel measurable family of distributions on $\R^{d}$ and $A \in \R_{>0}$ be such that, for every $l\in\N$, we have
\begin{equation}
\label{eq:reco-coherence}
\calN_{k,x} \fint_{\abs{h} \leq 2^{-k}} \abs{ (F_{x+h} - F_{x})(\phi^{\dilate{k+l}}_{x+h}) } \dif h
\leq
2^{l\alpha} A.
\end{equation}
Then, there exists a distribution $f \in \calD'(\R^{d})$ such that
\begin{equation} \label{eq:reco+1}
\calN_{k,x} \sup_{\xi \in \calB_{r}} \abs{ (f-F_x)(\xi_x^{\dilate{k}})}
\lesssim A.
\end{equation}
\end{theorem}
A few examples of quasinorms $\calN$ to which Theorem~\ref{thm:reconstruction} can be applied are listed below.
The uniqueness of the reconstruction $f$ in \eqref{eq:reco+1} does not seem to follow from the abstract properties of quasinorms $\calN$ stated above, and will be verified in the examples below separately, cf.\ \cite[Theorem 4.1]{MR4261666}.

We note that, as a consequence of \eqref{eq:fn}, if the distributions $F_{x}$ are in fact continuous functions and the function $(x,y) \mapsto F_{x}(y)$ is locally bounded near the diagonal, then $f$ is given by the function $f(x)=F_{x}(x)$.

\begin{proof}[Proof of Theorem~\ref{thm:reconstruction}]
In view of \eqref{eq:N-ae}, it suffices to verify $\calN_{k,x}H(k,x) \lesssim A$, where $H$ is the function defined in \eqref{eq:local-coherence}.
Since geometric series are summable in quasinormed spaces (this follows e.g.\ from the Aoki--Rolewicz theorem \cite[Lemma 1.1]{MR808777}, which is stated under additional hypotheses there, but in fact only uses quasisubadditivity), it suffices to obtain geometrically decaying bounds for the summands in the two sums over $l$ in \eqref{eq:local-coherence}.
For the first sum, this immediately follows from \eqref{eq:reco-coherence} and the assumption on $r$.
In the second sum, we have
\begin{equation}
\label{eq:g''-sum}
\begin{aligned}
\MoveEqLeft
\calN_{k,x}
\fint_{\abs{z-x}\leq 2^{-k}} \fint_{\abs{y-z}\leq 2^{-k-l}}
\abs{ (F_z-F_y)(\phi^{\dilate{k+l}}_y) } \dif y \dif z
\\ \text{by \eqref{eq:N-scaling}} &\lesssim
2^{-l\gamma} \calN_{k,x}
\fint_{\abs{y-x}\leq 2^{-k}}
\abs{ (F_x-F_y)(\phi^{\dilate{k}}_y) } \dif y
\\ \text{by \eqref{eq:reco-coherence}} &\lesssim
2^{-l\gamma} A.
\qedhere
\end{aligned}
\end{equation}

\end{proof}

\subsection{Besov spaces with $p\in [1,\infty]$}
\label{sec:Besov:p>1}
The prototypical example, previously studied in \cite{arxiv:2106.12528} (in the case $q\geq 1$), are Besov-type quasinorms
\begin{equation}
\label{eq:besov-norm:p>=1}
\calN_{k,x} H(k,x)
=
\ell^{q}_{k} 2^{\gamma k} L^{p}_{x} H(k,x).
\end{equation}
Here and later, we denote $\ell^{q}$ and $L^{p}$ norms by
\[
\ell^{q}_{k} H(k) := \bigl( \sum_{k=0}^{\infty} \abs{H(k)}^{q} \bigr)^{1/q},
\quad
L^{p}_{x} H(x) := \bigl( \int_{\R^{d}} \abs{H(x)}^{p} \dif x \bigr)^{1/p},
\]
when $p,q<\infty$, with the usual modification in the cases $p,q=\infty$.

\begin{lemma}
\label{eq:Besov:p>1:good-for-rec}
Let $q \in (0,\infty]$, $p\in [1,\infty]$, and $\gamma>0$.
Then the quasinorm \eqref{eq:besov-norm:p>=1} satisfies
\eqref{eq:N-scaling} (with same $\gamma$), \eqref{eq:N-monotone-conv}, and \eqref{eq:N-ae}.
\end{lemma}
Lemma~\ref{eq:Besov:p>1:good-for-rec} tells that Theorem~\ref{thm:reconstruction} can be applied with quasinorms \eqref{eq:besov-norm:p>=1}.
\begin{proof}
To see the condition \eqref{eq:N-scaling}, write
\begin{align*}
\MoveEqLeft
\calN_{k,x} \fint_{\abs{z-x}\leq 2^{-k}} H(k+l,z) \dif z
\\ &=
\ell^{q}_{k} 2^{\gamma k} L^{p}_{x} \fint_{\abs{h}\leq 2^{-k}} H(k+l,x+h) \dif h
\\ &\leq
\ell^{q}_{k} 2^{\gamma k} \fint_{\abs{h}\leq 2^{-k}} L^{p}_{x} H(k+l,x+h) \dif h
\\ &=
\ell^{q}_{k} 2^{\gamma k} L^{p}_{x} H(k+l,x+h)
\\ &\leq
2^{-\gamma l} \ell^{q}_{k} 2^{\gamma k} L^{p}_{x} H(k,x+h).
\end{align*}
Condition \eqref{eq:N-monotone-conv} follows from the monotone convergence theorem, and \eqref{eq:N-ae} is also immediate.
\end{proof}

\begin{lemma}[Uniqueness]
\label{eq:Besov:p>1:unique-rec}
Let $p,q,\gamma$ be as in Lemma~\ref{eq:Besov:p>1:good-for-rec}.
Then, for any family of distributions $F$, there exists at most one distribution $f$ such that \eqref{eq:reco+1} holds for the quasinorm \eqref{eq:besov-norm:p>=1}.
\end{lemma}
Lemma~\ref{eq:Besov:p>1:unique-rec} tells that the reconstruction produced in Theorem~\ref{thm:reconstruction} is unique for the quasinorms \eqref{eq:besov-norm:p>=1}.
\begin{proof}
Let $\xi \in \calB_{r}$ with $\int\xi=1$.
If $f,\tilde{f}$ both satisfy \eqref{eq:reco+1}, then
\[
\calN_{k,x} \abs{ (f-\tilde{f}) * \xi^{\dilate{k}} }(x)
< \infty.
\]
On the other hand,
\begin{equation}
\label{eq:5}
(f-\tilde{f}) * \xi^{\dilate{k}} \stackrel{\calD'}{\to} f-\tilde{f}.
\end{equation}
By definition of $\calN$, the left-hand side of \eqref{eq:5} converges to $0$ in $L^{p}(\R^{d})$, and it follows that $f=\tilde{f}$.
\end{proof}

\subsection{Besov spaces with $p\in (0,1)$}
\label{sec:Besov:p<1}
For $p<1$, the quasinorms \eqref{eq:besov-norm:p>=1} do not satisfy the condition \eqref{eq:N-scaling} (even with $l=0$), since Minkowski's inequality is false in $L^{p}$ in this range of $p$.

Instead, we consider the quasinorm
\begin{equation}
\label{eq:besov-norm:p<1}
\calN_{k,x} H(k,x)
=
\ell^{q}_{k} 2^{\nu k} L^{p}_{x} \fint_{y \in B(x,2^{-k})} H(k,y) \dif y.
\end{equation}
Interestingly, we can no longer consider all $\nu>0$.
In the one-dimensional case $d=1$, the range of allowed $\nu$ below corresponds to the range of the regularity exponents in the Besov space sewing lemma \cite{arxiv:2105.05978}.

\begin{lemma}
\label{lem:Besov:p<1}
Let $q \in (0,\infty]$, $p\in (0,1)$, and $\nu>d(1/p-1)$.
Then the quasinorm \eqref{eq:besov-norm:p<1} satisfies
\eqref{eq:N-scaling} (with $\gamma = \nu - d(1/p-1)$), \eqref{eq:N-monotone-conv}, and \eqref{eq:N-ae}.
\end{lemma}
\begin{proof}
Conditions \eqref{eq:N-monotone-conv} and \eqref{eq:N-ae} are easy to see.
Let us now show that \eqref{eq:N-scaling} holds with the claimed value of $\gamma$.
Indeed,
\begin{align*}
\calN_{k,x} \fint_{\abs{z-x} \leq 2^{-k}} H(k+l,z) \dif z
&\lesssim
\ell^{q}_{k} 2^{\nu k} L^{p}_{x} \fint_{\abs{z-x} \leq 2^{-k+1}} H(k+l,y) \dif y
\\ &=
\ell^{q}_{k\geq l} 2^{\nu (k-l)} L^{p}_{x} \fint_{\abs{y} \leq 2^{-k+l+1}} H(k,x+y) \dif y.
\end{align*}
We can cover the ball $B(0,2^{-k+l+1}) \subset \R^{d}$ by $N_{l} = O_{d}(2^{dl})$ many $2^{-k}$-balls $B_{1},\dotsc,B_{N_{l}}$.
Hence,
\begin{align*}
L^{p}_{x} \fint_{\abs{y} \leq 2^{-k+l+1}} H(k,x+y) \dif y
&\lesssim
N_{l}^{-1} L^{p}_{x} \sum_{j=1}^{N_{l}} \fint_{y \in B_{j}} H(k,x+y) \dif y
\\ &=
N_{l}^{-1} \Bigl( \int_{\R^{d}} \abs[\big]{ \sum_{j=1}^{N_{l}} \fint_{y \in B_{j}} H(k,x+y) \dif y }^{p} \dif x \Bigr)^{1/p}
\\ &\leq
N_{l}^{-1} \Bigl( \int_{\R^{d}} \sum_{j=1}^{N_{l}} \abs[\big]{ \fint_{y \in B_{j}} H(k,x+y) \dif y }^{p} \dif x \Bigr)^{1/p}
\\ &=
N_{l}^{-1} \Bigl( \sum_{j=1}^{N_{l}} \int_{\R^{d}} \abs[\big]{ \fint_{y \in B_{j}} H(k,x+y) \dif y }^{p} \dif x \Bigr)^{1/p}
\\ &=
N_{l}^{1/p-1} L^{p}_{x} \fint_{\abs{y} \leq 2^{-k}} H(k,x+y) \dif y.
\end{align*}
Hence, we see that
\[
\calN_{k,x} \fint_{\abs{z-x} \leq 2^{-k}} H(k+l,z) \dif z
\lesssim
2^{-\nu l} N_{l}^{1/p-1} \calN_{k,x} H(k,x)
\lesssim
2^{-\gamma l} \calN_{k,x} H(k,x).
\qedhere
\]
\end{proof}

\begin{lemma}[Uniqueness]
\label{eq:Besov:p<1:unique-rec}
Let $p,q,\gamma,\nu$ be as in Lemma~\ref{lem:Besov:p<1}.
Then, for any family of distributions $F$, there exists at most one distribution $f$ such that \eqref{eq:reco+1} holds for the quasinorm \eqref{eq:besov-norm:p<1}.
\end{lemma}
Lemma~\ref{eq:Besov:p<1:unique-rec} tells that the reconstruction produced in Theorem~\ref{thm:reconstruction} is unique in the cases covered by Lemma~\ref{lem:Besov:p<1}.
\begin{proof}
Let $\xi \in \calB_{r}$ with $\int\xi=1$.
If $f,\tilde{f}$ both satisfy \eqref{eq:reco+1}, then
\[
\calN_{k,x} \abs{ (f-\tilde{f}) * \xi^{\dilate{k}} }(x)
< \infty.
\]
As a consequence of the fact that the $\ell^{1}$ norm of a sequence is bounded by its $\ell^{p}$ norm, we obtain
\[
L^{1}_{x} \abs{ (f-\tilde{f}) * \xi^{\dilate{k}} }
\lesssim
2^{kd(1/p-1)} L^{p}_{x} \fint_{B(x,2^{-k})} \abs{ (f-\tilde{f}) * \xi^{\dilate{k}} }
\lesssim
2^{kd(1/p-1)} 2^{-\nu k}.
\]
Incidentally, this simple estimate is a special case of the reverse Young inequality, see \cite{MR2199372,MR1616143}.

The above estimate shows in particular that $\lim_{k\to\infty}(f-\tilde{f}) * \xi^{\dilate{k}} = 0$ in $L^{1}(\R^{d})$.
On the other hand, this sequence converges to $f-\tilde{f}$ in the sense of distributions, so that $f=\tilde{f}$.
\end{proof}

\subsection{Triebel--Lizorkin spaces}
The purpose of this section is to illustrate Theorem~\ref{thm:reconstruction} with an example of a function space norm that is not of Besov type.
One of the most natural quasinorms that one might consider is
\begin{equation}
\label{eq:triebel-lizorkin-norm}
\calN_{k,x} H(k,x) =
L^{p}_{x} \ell^{q}_{k} 2^{\gamma k} \bigl( \fint_{\abs{y} \leq 2^{-k}} H(k,x+y)^{q} \dif y \bigr)^{1/q}.
\end{equation}
A similar norm without the average over $y$ would be even simpler to look at.
\begin{lemma}
\label{lem:Triebel-Lizorkin}
Let $p \in (1,\infty)$, $q\in (1,\infty]$, and $\gamma >0$.
Then, the quasinorm \eqref{eq:triebel-lizorkin-norm} satisfies
\eqref{eq:N-scaling}, \eqref{eq:N-monotone-conv}, and \eqref{eq:N-ae}.
\end{lemma}
\begin{proof}
Conditions \eqref{eq:N-monotone-conv} and \eqref{eq:N-ae} are easy to see.
It remains to show \eqref{eq:N-scaling}.
We begin with the estimate
\begin{align*}
\MoveEqLeft
\calN_{k,x} \fint_{\abs{z-x} \leq 2^{-k}} H(k+l,z) \dif z
\\ &=
L^{p}_{x} \ell^{q}_{k} 2^{\gamma k} \Bigl( \fint_{\abs{y} \leq 2^{-k}} \bigl( \fint_{\abs{z} \leq 2^{-k}} H(k+l,x+y+z) \bigr)^{q} \dif y \Bigr)^{1/q}
\\ &\lesssim
L^{p}_{x} \ell^{q}_{k} 2^{\gamma k} \Bigl( \fint_{\abs{z} \leq 2^{-k+1}} H(k+l,x+z) \dif z \Bigr)
\\ &\lesssim
L^{p}_{x} \ell^{q}_{k} 2^{\gamma k} \Bigl( \fint_{\abs{z} \leq 2^{-k+1}} \bigl( \fint_{\abs{y} \leq 2^{-k-l}} H(k+l,x+y+z) \dif y \bigr) \dif z \Bigr)
\\ &\leq
L^{p}_{x} \ell^{q}_{k} 2^{\gamma k} \Bigl( \fint_{\abs{z} \leq 2^{-k+1}} \bigl( \fint_{\abs{y} \leq 2^{-k-l}} H(k+l,x+y+z)^{q} \dif y \bigr)^{1/q} \dif z \Bigr),
\end{align*}
with the usual modification if $q=\infty$.
By the Fefferman--Stein maximal inequality, see e.g.\ \cite[Theorem 3.2.28]{MR3617205}, this is bounded by
\[
L^{p}_{x} \ell^{q}_{k} 2^{\gamma k} \bigl( \fint_{\abs{y} \leq 2^{-k-l}} H(k+l,x+y)^{q} \dif y \bigr)^{1/q}
\leq
2^{-\gamma l} L^{p}_{x} \ell^{q}_{k} 2^{\gamma k} \bigl( \fint_{\abs{y} \leq 2^{-k}} H(k,x+y)^{q} \dif y \bigr)^{1/q}.
\qedhere
\]
\end{proof}

\begin{lemma}[Uniqueness]
\label{lem:Triebel-Lizorkin:unique-rec}
Let $p,q,\gamma$ be as in Lemma~\ref{lem:Triebel-Lizorkin}.
Then, for any family of distributions $F$, there exists at most one distribution $f$ such that \eqref{eq:reco+1} holds for the quasinorm \eqref{eq:triebel-lizorkin-norm}.
\end{lemma}
\begin{proof}
If $f,\tilde{f}$ are two distributions for which \eqref{eq:triebel-lizorkin-norm} holds and $\xi \in \calB_{r}$ with $\int \xi = 1$, then
\[
\lim_{k\to\infty} L^{p}_{x} \bigl( \fint_{\abs{y} \leq 2^{-k}} \abs{(f-\tilde{f})*\xi^{\dilate{k}}}^{q} \dif y \bigr)^{1/q}
= 0.
\]
This implies in particular that $(f-\tilde{f})*\xi^{\dilate{k}} \to 0$ in $L^{1}(\Omega)$ for any bounded measurable set $\Omega \subset \R^{d}$, so that $f=\tilde{f}$.
\end{proof}

\section{Proof of Theorem~\ref{thm:local-reconstruction}}
\label{sec:proof}
The entirety of this section is occupied by the proof of Theorem~\ref{thm:local-reconstruction}.

\subsection{Approximating functions}
In this subsection, we recall the construction of mollifiers and approximating functions in \cite{MR4261666}.

By \cite[Lemma 8.1]{MR4261666}, there exists a linear combination of dilated functions $\phi^{(0)},\dotsc,\phi^{(\tilde{r})}$ that has integral $1$ and vanishing moments of orders $1,\dotsc,\tilde{r}$.
Replacing $\phi$ by this linear combination in \eqref{eq:local-coherence} increases the value of $H$ at most by a multiplicative constant (depending on $r$), so we may assume that
\[
\int_{\R^d} \phi(y) \dif y = 1
\quad\text{and}\quad
\int_{\R^d} y^\alpha \phi(y) \dif y = 0
\quad \text{for } 1 \le \abs{\alpha} \le \tilde{r}.
\]
We use the mollifiers $\rho := \phi^{\dilate{1}} * \phi$, so that
\begin{equation} \label{eq:diffconv}
\rho^{\dilate{1}} - \rho = \phi^{(1)} * \psi,
\qquad \text{where} \qquad \psi := \phi^{\dilate{2}} - \phi.
\end{equation}
Note that $\int\rho=\int \phi^{\dilate{1}} \int \phi=1$ and
\begin{equation} \label{eq:checkvarphi2}
\rho^{\dilate{n+1}} - \rho^{\dilate{n}}
= (\rho^{\dilate{1}} - \rho)^{\dilate{n}}
= \phi^{\dilate{n+1}} * \psi^{\dilate{n}} .
\end{equation}
This will be used to compare convolutions with $\rho^{\dilate{n+1}}$ and $\rho^{\dilate{n}}$.

The mollifiers $\rho$ are used to define the functions \eqref{eq:fn}.
The integrability of the functions \eqref{eq:fn} will follow from the bound \eqref{eq:1} below and smoothness of $z \mapsto F_{x}(\rho^{\dilate{n}}_{z})$.

A novelty of our presentation with respect to \cite{MR4261666} is that we deduce both the convergence of the sequence \eqref{eq:fn} and the estimate \eqref{eq:reco+1} from the same set of estimates summarized in \eqref{eq:telescoping-bound}.

\subsection{Telescoping sum}
We will write the sequence
\begin{equation}\label{neq:tildefn}
f_{x,n}(z):=
(f_{n} - F_{x}*\tilde{\rho}^{\dilate{n}})(z)
=
(F_z - F_x)(\rho^{\dilate{n}}_z)
\end{equation}
as a telescoping sum.
The difference of consecutive functions in the sequence \eqref{neq:tildefn} is
\begin{align*}
f_{x,n+1}(z) - f_{x,n}(z)
&=
(F_z-F_{x})(\rho^{\dilate{n+1}}_z - \rho^{\dilate{n}}_z)
\\ &=
\int_{\R^d} (F_{z}-F_{x})(\phi^{\dilate{n+1}}_y) \psi^{\dilate{n}}_z(y) \dif y
\\ &=
\int_{\R^d} (F_y-F_{x})(\phi^{\dilate{n+1}}_y) \psi^{\dilate{n}}(y-z) \dif y
\\ &\quad+
\int_{\R^d} (F_z-F_{y})(\phi^{\dilate{n+1}}_y) \psi^{\dilate{n}}(y-z) \dif y
\\ &=:
g_{x,n}'(z) + g_{n}''(z).
\end{align*}
In Section~\ref{sec:estimates}, we will show
\begin{equation}
\label{eq:telescoping-bound}
\sup_{\xi\in\calB_r} \abs{f_{x,k}(\xi_x^{\dilate{k}})}
+ \sum_{l\geq 0} \abs{g_{x,k+l}'(\xi_x^{\dilate{k}})}
+ \sum_{l\geq 0} \abs{g_{k+l}''(\xi_x^{\dilate{k}})}
\lesssim H(k,x).
\end{equation}
Since also $\lim_{n\to\infty} F_{x} * \tilde{\rho}^{\dilate{n}} = F_{x}$ in the sense of distributions, it follows that the limit
\begin{align*}
\xi \mapsto f(\xi_{x}^{\dilate{k}})
&:=
\lim_{n\to \infty} f_{n}(\xi_x^{\dilate{k}})
\\ &=
\lim_{n\to \infty} (f_{x,n} + F_{x}* \tilde{\rho}^{\dilate{n}})(\xi_x^{\dilate{k}})
\\ &=
\lim_{n\to \infty} (f_{x,k} + \sum_{l=0}^{n-k-1} (f_{x,k+l+1}-f_{x,k+l}))(\xi_x^{\dilate{k}}) + F_{x}(\xi_x^{\dilate{k}})
\\ &=
f_{x,k}(\xi_x^{\dilate{k}})
+ \sum_{l\geq 0} g_{x,k+l}'(\xi_x^{\dilate{k}})
+ \sum_{l\geq 0} g_{k+l}''(\xi_x^{\dilate{k}})
+ F_{x}(\xi_x^{\dilate{k}})
\end{align*}
exists and is bounded by \eqref{eq:telescoping-bound} for $\xi \in \calB_{r}$.

\subsection{Estimates} \label{sec:estimates}
It remains to show \eqref{eq:telescoping-bound}.
Let $\xi\in\calB_r$ be arbitrary.
We start with the sums over $l$, because the first term in \eqref{eq:telescoping-bound} will be estimated by quantities that appear in the bounds for these sums.

\begin{proof}[Estimate for $g'$]
By definition,
\begin{equation*}
\begin{split}
g_{x,k+l}'(\xi_{x}^{\dilate{k}})
&=
\int_{\R^d} \int_{\R^d} (F_y - F_x)(\phi^{\dilate{k+l+1}}_y) \psi^{\dilate{k+l}}(y-z) \xi_x^{\dilate{k}}(z) \dif y \dif z
\\ & =
\int_{\R^d} (F_y - F_x)(\phi^{\dilate{k+l+1}}_y) (\psi^{\dilate{k+l}} * \xi_x^{\dilate{k}})(y) \dif y.
\end{split}
\end{equation*}
As in \cite[Lemma 9.2]{MR4261666}, using Taylor's formula of order $\tilde{r}$ for $\xi$, we see that $\abs{\psi^{\dilate{k+l}} * \xi_x^{\dilate{k}}} \lesssim 2^{-lr+kd}$.
Moreover, this convolution is supported on $B(0,2^{-k})$.
It follows that
\begin{equation} \label{eq:g'-pointwise}
\abs{g_{x,k+l}'(\xi_{x}^{\dilate{k}})}
\lesssim 2^{-lr}
\fint_{\abs{y-x} \leq 2^{-k}} \abs{ (F_y - F_x)(\phi^{\dilate{k+l+1}}_y) } \dif y.
\end{equation}
This is the $(l+1)$-th summand in the first sum in \eqref{eq:local-coherence}.
\end{proof}

\begin{proof}[Estimate for $g''$]
Using the support conditions on $\psi$ and $\xi$, we obtain
\begin{equation} \label{eq:g''-pointwise}
\abs{g''_{k+l}(\xi_x^{\dilate{k}})}
\lesssim
\fint_{\abs{z-x}\leq 2^{-k-1}} \fint_{\abs{y-z}\leq 2^{-k-l-1}}
\abs{ (F_z-F_y)(\phi^{\dilate{k+l+1}}_y) } \dif y \dif z.
\end{equation}
This is bounded by the $(l+1)$-th summand in the second sum in \eqref{eq:local-coherence}.
\end{proof}

\begin{proof}[Estimate for $f_{x,k}$]
By the support condition on $\xi$, we have
\[
\abs{f_{x,k}(\xi_x^{\dilate{k}})}
\lesssim
\fint_{z\in B(x,2^{-k-1})} \abs{ f_{x,k}(z) } \dif z.
\]
For any $l\geq 0$, we have
\begin{equation}
\label{eq:1}
\begin{aligned}
\MoveEqLeft
\fint_{\abs{z-x}\leq 2^{-k-1}} \abs{ f_{x,k+l}(z) } \dif z
\\ &=
\fint_{\abs{z-x}\leq 2^{-k-1}} \abs[\Big]{ \int_{\R^d} (F_z - F_x)(\phi^{\dilate{k+l}}_y) \,
\phi^{\dilate{k+l+1}}(y-z) \dif y } \dif z
\\ &\lesssim
\fint_{\abs{z-x}\leq 2^{-k-1}} \fint_{\abs{y-z} \le 2^{-k-l-2}} \abs{(F_z - F_x)(\phi^{\dilate{k}}_y)} \dif y \dif z
\\ &\leq
\fint_{\abs{z-x}\leq 2^{-k-1}} \fint_{\abs{y-z} \le 2^{-k-l-2}}
\Bigl( \abs{(F_z - F_y)(\phi^{\dilate{k}}_y)}
+ \abs{(F_y - F_x)(\phi^{\dilate{k}}_y)} \Bigr) \dif y \dif z.
\end{aligned}
\end{equation}
This is bounded by the $l$-th summands in \eqref{eq:local-coherence}.
This shows in particular the local integrability of the functions~\eqref{eq:fn}, and the $l=0$ case finishes the proof of \eqref{eq:telescoping-bound}, and therefore also the proof of Theorem~\ref{thm:reconstruction}.
\end{proof}

\section{Besov sewing and reconstruction}
\label{sec:sewing}
In this section, we show how Theorem~\ref{thm:reconstruction} can be used to recover the main estimate in the Besov space sewing lemma \cite[Theorem 3.1]{arxiv:2105.05978} in the case $p\geq 1$, assuming an additional qualitative regularity condition on the data.

For two-parameter processes $A : \R \times \R \to \R$, we use the quasinorms
\begin{equation}\label{eq:Besov-A}
\mathbb{B}^{\eta}_{p,q} A
:= \ell^{q}_{k} 2^{\eta k} \sup_{\abs{h}\leq 2^{-k}} L^{p}_{x} \abs{A(x,x+h)}.
\end{equation}
For three-parameter processes $G : \R \times \R \times \R \to \R$, we use the quasinorms
\begin{equation}\label{eq:Besov-delta-A}
\bar{B}^{\eta}_{p,q} G
:= \ell^{q}_{k} 2^{\eta k} \sup_{\abs{y'},\abs{y''}\leq 2^{-k+1}} L^{p}_{x} \abs{G(x,x+y',x+y'')}.
\end{equation}

The increments of a one-parameter process $g: \R\to\R$ and a two-parameter process $A :\R\times\R\to\R$ are defined by
\[
\delta g(s,t) = g(t)-g(s)
\quad\text{and}\quad
\delta A(s,u,t) = A(s,t)-A(s,u)-A(u,t),
\]
respectively.

\begin{theorem}[Smooth Besov sewing for $p\geq 1$, special case of {\cite[Theorem 3.1]{arxiv:2105.05978}}]
\label{thm:sewing}
Let $q \in (0,\infty]$, $p \in [1,\infty]$, and $\eta > 1$.
Let $A : \R\times\R \to \R$ be a smooth function with $\bar{B}^{\eta}_{p,q} \delta A < \infty$.
Then, there exists a function $g : \R \to \R$ with
\begin{equation}
\label{eq:sewing-bd}
\mathbb{B}^{\eta}_{p,q} (\delta g - A)
\lesssim
\bar{B}^{\eta}_{p,q} \delta A.
\end{equation}
\end{theorem}
In \cite{arxiv:2105.05978}, most processes are defined on increasing pairs or triples of indices.
Given an $A$ defined on increasing pairs and vanishing on the diagonal, one can extend it to an antisymmetric function on $\R\times\R$, in the sense that $A(s,t)=-A(t,s)$.
Note that if $A$ is antisymmetric, then $\delta A$ is also antisymmetric in all its arguments.
The norms \eqref{eq:Besov-A} and \eqref{eq:Besov-delta-A} of such antisymmetric extensions are equivalent to the norms considered in \cite{arxiv:2105.05978}.

However, Theorem~\ref{thm:sewing} does not fully recover \cite[Theorem 3.1]{arxiv:2105.05978}, because in general it does not seem possible to approximate a function in the \eqref{eq:Besov-A} norm by smooth functions.
On the other hand, by using Theorem~\ref{thm:reconstruction}, we attempt to construct $g$ using smoothed versions of $A$.

Another indication that smoothing occuring in Theorem~\ref{thm:reconstruction} is harmful for recovering the sewing lemma in terms of the quasinorms \eqref{eq:Besov-A} and \eqref{eq:Besov-delta-A} is that our argument does not easily extend to $p \in (0,1)$, because it is not possible to pull the integral out of the $L^{p}$ quasinorm in \eqref{eq:2}.
This can be circumvented by considering different Besov-like quasinorms, as in \eqref{eq:besov-norm:p<1}, but there does not seem to be any advantage in doing that over running the sewing argument in \cite{arxiv:2105.05978} directly.

\begin{proof}[Proof of Theorem~\ref{thm:sewing}]
We will use Theorem~\ref{thm:reconstruction} with $\gamma=\eta-1$, the quasinorm \eqref{eq:besov-norm:p>=1}, and the family of distributions $F_{x} = D_{2}A(x,\cdot)$.
The function $\phi$ can be chosen arbitrarily with support in $B(0,1/2)$.
First, we have to verify the hypothesis \eqref{eq:reco-coherence}.
Using that $\int D\phi = 0$, we obtain
\begin{align*}
\MoveEqLeft
\abs[\Big]{ (F_{x}-F_{x+h})(\phi_{x+h}^{\dilate{k}}) }
\\ &=
\abs[\Big]{ (A(x,\cdot)-A(x+h,\cdot))(D \phi_{x+h}^{\dilate{k}}) }
\\ &=
2^{k} \abs[\Big]{ (A(x,\cdot)-A(x+h,\cdot))((D\phi)_{x+h}^{\dilate{k}}) }
\\ &=
2^{k} \abs[\Big]{ \int (A(x,y)-A(x+h,y)-A(x,x+h))(D\phi)_{x+h}^{\dilate{k}}(y) \dif y }
\\ &\lesssim
2^{k} \fint_{B(x+h,2^{-k-1})} \abs{\delta A(x,x+h,y)} \dif y.
\end{align*}
Hence,
\begin{equation}
\label{eq:2}
\begin{split}
\MoveEqLeft
\ell^{q}_{k} 2^{\gamma k} L^{p}_{x} \fint_{\abs{h} \leq 2^{-k}} \abs{ (F_{x+ h} - F_{x})(\phi^{\dilate{k+l}}_{x+ h}) } \dif h
\\ &\lesssim
\ell^{q}_{k} 2^{(\gamma+1)k+l} L^{p}_{x} \fint_{\abs{h} \leq 2^{-k}} \fint_{y \in B(x+h,2^{-k-l-1})} \abs{\delta A(x,x+h,y)} \dif y \dif h
\\ &\leq
2^{l} \ell^{q}_{k} 2^{(\gamma+1)k} \sup_{\abs{h}, \abs{h'} \leq 2^{-k+1}}
L^{p}_{x} \abs{\delta A(x,x+h,x+h')}
\\ &=
2^{l} \bar{B}^{\eta}_{p,q} \delta A.
\end{split}
\end{equation}
This shows \eqref{eq:reco-coherence} with $\alpha=1$.

Theorem~\ref{thm:reconstruction} with any $r>1$ gives us a distribution $f$ such that
\begin{equation}
\label{eq:4}
\ell^{q}_{k} 2^{\gamma k} L^{p}_{x}
\underbrace{ \sup_{\xi \in \calB_{r}} \abs{(f-D_{2}A(x,\cdot))(\xi_{x}^{\dilate{k}})} }_{=:\Delta(k,x)}
< \infty.
\end{equation}

In the proof of Theorem~\ref{thm:reconstruction}, the distribution $f$ was constructed as the distributional limit of the sequence of functions \eqref{eq:fn}.
In the current setting, the sequence \eqref{eq:fn} converges locally uniformly to the function $x \mapsto D_{2}A(x,x)$, so that $f$ coincides with this continuous function.
Let $g$ be any antiderivative of $f$ (the qualitative hypothesis that $A$ is smooth is needed to construct this antiderivative).
It remains to show \eqref{eq:sewing-bd}.

Let $\tilde{\chi} : \R \to [0,1]$ be a function supported on $[0,0.6)$ and smooth on $(0,1)$ such that $\tilde{\chi}(\theta) + \tilde{\chi}(1-\theta) = \one_{[0,1]}(\theta)$.
Let $\chi(\theta) := \tilde{\chi}(\theta) - \tilde{\chi}(2\theta)$, so that $\supp \chi \subseteq [0.2,0.6]$ and
\[
\sum_{l\in\N} \chi(2^{l} \theta) + \chi(2^{l} (1-\theta)) = \one_{(0,1)}(\theta).
\]
By the fundamental theorem of calculus, we have
\begin{align}
\MoveEqLeft \notag
g(x+h)-g(x)-A(x,x+h)
\\ &= \label{eq:near0}
\abs{h} \int_{\R} (f(x+\theta h) - D_{2}A(x,x+\theta h)) \tilde{\chi}(\theta) \dif\theta
\\ &+ \label{eq:near1}
\abs{h} \int_{\R} (f(x+\theta h) - D_{2}A(x,x+\theta h)) \tilde{\chi}(1-\theta) \dif\theta.
\end{align}
For $2^{-k-1} < \abs{h} \leq 2^{-k}$, we estimate
\begin{align*}
\abs{\eqref{eq:near0}}
& \leq
\sum_{l\geq 0} \abs[\Big]{ \int_{\R} (f- D_{2}A(x,\cdot)) \chi(2^{l}(\cdot-x)/h) }
\\ &\lesssim
\sum_{l\geq 0} 2^{-l} \abs{h} \Delta(k+l,x).
\end{align*}
In the second summand, we split
\begin{align}
\eqref{eq:near1}
&= \label{eq:near1-c}
\abs{h} \int_{\R} (f(x+\theta h) - D_{2}A(x+h,x+\theta h)) \tilde{\chi}(1-\theta) \dif\theta
\\ &+ \label{eq:near1-d}
\abs{h} \int_{\R} (D_{2}A(x+h,x+\theta h) - D_{2}A(x,x+\theta h)) \tilde{\chi}(1-\theta) \dif\theta.
\end{align}
Similarly as before, for $2^{-k-1} < \abs{h} \leq 2^{-k}$, we obtain
\[
\abs{\eqref{eq:near1-c}}
\lesssim
\sum_{l\geq 0} 2^{-l} \abs{h} \Delta(k+l,x+h).
\]
In the last term, we use partial integration:
\begin{align*}
\abs{\eqref{eq:near1-d}}
&=
\abs[\Big]{ \int_{\R} (A(x+h,x+\theta h) - A(x,x+\theta h)) D\tilde{\chi}(1-\theta) \dif\theta }
\\ &=
\abs[\Big]{ \int_{\R} \delta A(x+h,x,x+\theta h) D\tilde{\chi}(1-\theta) \dif\theta }
\\ &\leq
\int_{\R} \abs{\delta A(x+h,x,x+(1-\theta) h)} \dif \abs{D\tilde{\chi}}(\theta),
\end{align*}
where we use that the distributional derivative $D\tilde{\chi}$ is a finite measure.
Collecting the bounds for \eqref{eq:near0}, \eqref{eq:near1-c}, and \eqref{eq:near1-d}, we obtain
\begin{align*}
\MoveEqLeft
\ell^{q}_{k} 2^{\eta k} \sup_{\abs{h} \leq 2^{-k}} L^{p} \abs{g(x+h)-g(x)-A(x,x+h)}
\\ &\lesssim
\ell^{q}_{k} 2^{\eta k} \sup_{\tilde{k}\geq k} \sup_{2^{-\tilde{k}-1} < \abs{h} \leq 2^{-\tilde{k}}} L^{p}_{x}
\sum_{l\geq 0} 2^{-l-\tilde{k}} ( \Delta(\tilde{k}+l,x) + \Delta(\tilde{k}+l,x+h) )
\\ &+
\ell^{q}_{k} 2^{\eta k} \sup_{\abs{h} \leq 2^{-k}} L^{p}_{x}
\int_{\R} \abs{\delta A(x+h,x,x+(1-\theta) h)} \dif \abs{D\tilde{\chi}}(\theta)
\\ &\lesssim
\ell^{q}_{k} 2^{\eta k} \sup_{\tilde{k}\geq k} \sum_{l\geq 0} 2^{-l-\tilde{k}} L^{p}_{x} \Delta(\tilde{k}+l,x)
\\ &+
\ell^{q}_{k} 2^{\eta k} \sup_{\abs{h}, \abs{\tilde{h}} \leq 2^{-k}} L^{p}_{x}
\abs{\delta A(x+h,x,x+\tilde{h})}.
\end{align*}
The last summand is bounded directly by $\bar{B}^{\eta}_{p,q} \delta A$.
In the first summand, we have
\[
\sup_{\tilde{k}\geq k} \sum_{l\geq 0} 2^{-l-\tilde{k}} L^{p}_{x} \Delta(\tilde{k}+l,x)
=
2^{-k} \sup_{\tilde{k}\geq k} \sum_{l\geq \tilde{k}-k} 2^{-l} L^{p}_{x} \Delta(k+l,x)
=
2^{-k} \sum_{l\geq 0} 2^{-l} L^{p}_{x} \Delta(k+l,x).
\]
Since geometric series are summable in $\ell^{q}$ for any $q>0$, it remains to obtain an exponentially decreasing (in $l\geq 0$) bound for
\[
\ell^{q}_{k} 2^{\eta k} L^{p}_{x} 2^{-k-l} \Delta(k+l,x).
\]
This follows directly from \eqref{eq:4}.
\end{proof}

\section{Negative regularity}
\label{sec:negative-regularity}
In \cite{MR4261666}, a version of the reconstruction theorem is formulated also for $\gamma<0$.
To obtain a corresponding version of Theorem~\ref{thm:local-reconstruction}, one can replace the second sum over $l\geq 0$ in \eqref{eq:local-coherence} by a sum over $l\leq 0$.
The bound \eqref{eq:telescoping-bound} is then replaced by
\begin{equation}
\label{eq:telescoping-bound:gamma<0}
\sup_{\xi\in\calB_r} \abs{f_{x,k}(\xi_x^{\dilate{k}})}
+ \sum_{l\geq 0} \abs{g_{x,k+l}'(\xi_x^{\dilate{k}})}
+ \sum_{l=-k}^{0} \abs{g_{k+l}''(\xi_x^{\dilate{k}})}
\lesssim H(k,x),
\end{equation}
and the formula for $f$ is different:
\begin{align*}
f(\xi_{x}^{\dilate{k}})
&:=
\lim_{n\to\infty} (f_{n} - \sum_{m=1}^{n-1} g_{m}'')(\xi_{x}^{\dilate{k}})
\\ &=
\lim_{n\to\infty} (f_{x,n} + F_{x}*\tilde{\rho}^{\dilate{n}} - \sum_{m=1}^{n-1} g_{m}'')(\xi_{x}^{\dilate{k}})
\\ &=
\lim_{n\to\infty} (f_{x,k} + \sum_{l=0}^{n-k-1} (f_{x,k+l+1}-f_{x,k+l}) - \sum_{m=1}^{n-1} g_{m}'')(\xi_{x}^{\dilate{k}}) + F_{x}(\xi_{x}^{\dilate{k}})
\\ &=
\lim_{n\to\infty} (f_{x,k} + \sum_{l=0}^{n-k-1} (g_{x,k+l}'+g_{k+l}'') - \sum_{m=1}^{n-1} g_{m}'')(\xi_{x}^{\dilate{k}}) + F_{x}(\xi_{x}^{\dilate{k}})
\\ &=
(f_{x,k} + \sum_{l\geq 0} g_{x,k+l}' - \sum_{m=1}^{k-1} g_{m}'')(\xi_{x}^{\dilate{k}}) + F_{x}(\xi_{x}^{\dilate{k}}).
\end{align*}
It is possible to formulate a corresponding version of Theorem~\ref{thm:reconstruction},
which then applies for instance to Besov spaces of negative regularity (with quasinorm $\calN$ given by \eqref{eq:besov-norm:p>=1} with $\gamma<0$).
However, a distribution satisfying \eqref{eq:reco+1} will not be unique in this case, for instance, this condition is preserved under adding a point mass to $f$.

\printbibliography
\end{document}